\documentclass[12pt]{amsart}
\usepackage{amssymb,latexsym}    

\input{xypic}
\newcommand{\Mdef}[2]{\newcommand{#1}{\relax \ifmmode #2 \else $#2$\fi}}

\newcommand{\sm }{\wedge}

\newcommand{\map}{\mathrm{map}}
\newcommand{\Hom}{\mathrm{Hom}}

\Mdef{\bhom}{\mathbf{\hat{H}om}}

\Mdef{\Mod}{\mathrm{mod}}

\newcommand{\st}{\; | \;}

\newtheorem{thm}{Theorem}[section]
\newtheorem{theorem}{Theorem}[section]

\newtheorem{proposition}[thm]{Proposition}

\newtheorem{corollary}[thm]{Corollary}

\theoremstyle{definition}

\newtheorem{definition}[thm]{Definition}

\newtheorem{example}[thm]{Example}

\newtheorem{remark}[thm]{Remark}

\newcommand{\qqed}{\qed \\[1ex]}
\renewenvironment{proof}[1][\hspace*{-.8ex}]{\noindent {\bf Proof #1:\;}}{\qqed}

\Mdef{\PH} {\Phi^H}
\Mdef{\PK} {\Phi^K}
\Mdef{\PL} {\Phi^L}
\Mdef{\PT} {\Phi^{\T}}

\Mdef{\ef}{E{\cF}_+}
\Mdef{\etf}{\widetilde{E}{\cF}}
\Mdef{\eg}{E{G}_+}
\Mdef{\etg}{\tilde{E}{G}}

\Mdef{\infl}{\mathrm{inf}}
\Mdef{\defl}{\mathrm{def}}
\Mdef{\res}{\mathrm{res}}
\Mdef{\ind}{\mathrm{ind}}
\Mdef{\coind}{\mathrm{coind}}

\Mdef{\univ}{\mathcal{U}}

\Mdef{\Fp}{\mathbb{F}_p}
\Mdef{\Zpinfty}{\Z /p^{\infty}}
\Mdef{\Zpadic}{\Z_p^{\wedge}}

\newcommand{\bi}{\begin{itemize}}
\newcommand{\be}{\begin{enumerate}}
\newcommand{\bc}{\begin{center}}
\newcommand{\bd}{\begin{description}}
\newcommand{\ei}{\end{itemize}}
\newcommand{\ee}{\end{enumerate}}
\newcommand{\ec}{\end{center}}
\newcommand{\ed}{\end{description}}

\newcommand{\adjunction}[4]{
\diagram
#1:#2 \rrto<0.7ex> &&
#3  \llto<0.7ex> :#4 
\enddiagram}
\newcommand{\lra}{\longrightarrow}

\newcommand{\iso}{\cong}

\Mdef{\we}{\mathbf{we}}
\Mdef{\fib}{\mathbf{fib}}
\Mdef{\cof}{\mathbf{cof}}
\Mdef{\BI}{\mathcal{BI}}

\Mdef{\A}{\mathbb{A}}
\Mdef{\B}{\mathbb{B}}
\Mdef{\C}{\mathbb{C}}
\Mdef{\D}{\mathbb{D}}
\Mdef{\E}{\mathbb{E}}
\Mdef{\T}{\mathbb{T}}
\Mdef{\F}{\mathbb{F}}
\Mdef{\G}{\mathbb{G}}
\Mdef{\I}{\mathbb{I}}
\Mdef{\N}{\mathbb{N}}
\Mdef{\Q}{\mathbb{Q}}
\Mdef{\R}{\mathbb{R}}
\Mdef{\bbS}{\mathbb{S}}
\Mdef{\Z}{\mathbb{Z}}

\Mdef{\bA}{\mathbb{A}}
\Mdef{\bB}{\mathbb{B}}
\Mdef{\bC}{\mathbb{C}}
\Mdef{\bD}{\mathbb{D}}
\Mdef{\bE}{\mathbb{E}}
\Mdef{\bF}{\mathbb{F}}
\Mdef{\bG}{\mathbb{G}}
\Mdef{\bH}{\mathbb{H}}
\Mdef{\bI}{\mathbb{I}}
\Mdef{\bJ}{\mathbb{J}}
\Mdef{\bK}{\mathbb{K}}
\Mdef{\bL}{\mathbb{L}}
\Mdef{\bM}{\mathbb{M}}
\Mdef{\bN}{\mathbb{N}}
\Mdef{\bO}{\mathbb{O}}
\Mdef{\bP}{\mathbb{P}}
\Mdef{\bQ}{\mathbb{Q}}
\Mdef{\bR}{\mathbb{R}}
\Mdef{\bS}{\mathbb{S}}
\Mdef{\bT}{\mathbb{T}}
\Mdef{\bU}{\mathbb{U}}
\Mdef{\bV}{\mathbb{V}}
\Mdef{\bW}{\mathbb{W}}
\Mdef{\bX}{\mathbb{X}}
\Mdef{\bY}{\mathbb{Y}}
\Mdef{\bZ}{\mathbb{Z}}

\Mdef{\cA}{\mathcal{A}}
\Mdef{\cB}{\mathcal{B}}
\Mdef{\cC}{\mathcal{C}}
\Mdef{\mcD}{\mathcal{D}} 
\Mdef{\cE}{\mathcal{E}}
\Mdef{\cF}{\mathcal{F}}
\Mdef{\cG}{\mathcal{G}}
\Mdef{\mcH}{\mathcal{H}} 
\Mdef{\cI}{\mathcal{I}}
\Mdef{\cJ}{\mathcal{J}}
\Mdef{\cK}{\mathcal{K}}
\Mdef{\mcL}{\mathcal{L}}

\Mdef{\cM}{\mathcal{M}}
\Mdef{\cN}{\mathcal{N}}
\Mdef{\cO}{\mathcal{O}}
\Mdef{\cP}{\mathcal{P}}
\Mdef{\cQ}{\mathcal{Q}}
\Mdef{\mcR}{\mathcal{R}}
\Mdef{\cS}{\mathcal{S}}
\Mdef{\cT}{\mathcal{T}}
\Mdef{\cU}{\mathcal{U}}
\Mdef{\cV}{\mathcal{V}}
\Mdef{\cW}{\mathcal{W}}
\Mdef{\cX}{\mathcal{X}}
\Mdef{\cY}{\mathcal{Y}}
\Mdef{\cZ}{\mathcal{Z}}

\Mdef{\At}{\tilde{A}}
\Mdef{\Bt}{\tilde{B}}
\Mdef{\Ct}{\tilde{C}}
\Mdef{\Et}{\tilde{E}}
\Mdef{\Ht}{\tilde{H}}
\Mdef{\Kt}{\tilde{K}}
\Mdef{\Lt}{\tilde{L}}
\Mdef{\Mt}{\tilde{M}}
\Mdef{\Nt}{\tilde{N}}
\Mdef{\Pt}{\tilde{P}}

\Mdef{\tA}{\tilde{A}}
\Mdef{\tB}{\tilde{B}}
\Mdef{\tC}{\tilde{C}}
\Mdef{\tE}{\tilde{E}}
\Mdef{\tH}{\tilde{H}}
\Mdef{\tK}{\tilde{K}}
\Mdef{\tL}{\tilde{L}}
\Mdef{\tM}{\tilde{M}}
\Mdef{\tN}{\tilde{N}}
\Mdef{\tP}{\tilde{P}}

\Mdef{\ft}{\tilde{f}}
\Mdef{\xt}{\tilde{x}}
\Mdef{\yt}{\tilde{y}}

\Mdef{\Ab}{\overline{A}}
\Mdef{\Bb}{\overline{B}}
\Mdef{\Cb}{\overline{C}}
\Mdef{\Db}{\overline{D}}
\Mdef{\Eb}{\overline{E}}
\Mdef{\Fb}{\overline{F}}
\Mdef{\Gb}{\overline{G}}
\Mdef{\Hb}{\overline{H}}
\Mdef{\Ib}{\overline{I}}
\Mdef{\Jb}{\overline{J}}
\Mdef{\Kb}{\overline{K}}
\Mdef{\Lb}{\overline{L}}
\Mdef{\Mb}{\overline{M}}
\Mdef{\Nb}{\overline{N}}
\Mdef{\Ob}{\overline{O}}
\Mdef{\Pb}{\overline{P}}
\Mdef{\Qb}{\overline{Q}}
\Mdef{\Rb}{\overline{R}}
\Mdef{\Sb}{\overline{S}}
\Mdef{\Tb}{\overline{T}}
\Mdef{\Ub}{\overline{U}}
\Mdef{\Vb}{\overline{V}}
\Mdef{\Wb}{\overline{W}}
\Mdef{\Xb}{\overline{X}}
\Mdef{\Yb}{\overline{Y}}
\Mdef{\Zb}{\overline{Z}}

\Mdef{\db}{\overline{d}}
\Mdef{\hb}{\overline{h}}
\Mdef{\qb}{\overline{q}}
\Mdef{\rb}{\overline{r}}
\Mdef{\tb}{\overline{t}}
\Mdef{\ub}{\overline{u}}
\Mdef{\vb}{\overline{v}}

\Mdef{\hc}{\hat{c}}
\Mdef{\he}{\hat{e}}
\Mdef{\hf}{\hat{f}}
\Mdef{\hA}{\hat{A}}
\Mdef{\hH}{\hat{H}}
\Mdef{\hJ}{\hat{J}}
\Mdef{\hM}{\hat{M}}
\Mdef{\hP}{\hat{P}}
\Mdef{\hQ}{\hat{Q}}

\Mdef{\thetab}{\overline{\theta}}
\Mdef{\phib}{\overline{\phi}}

\Mdef{\uA}{\underline{A}}
\Mdef{\uB}{\underline{B}}
\Mdef{\uC}{\underline{C}}
\Mdef{\uD}{\underline{D}}

\Mdef{\bolda}{\mathbf{a}}
\Mdef{\boldb}{\mathbf{b}}
\Mdef{\boldD}{\mathbf{D}}

\Mdef{\fm}{\frak{m}}
\Mdef{\fp}{\frak{p}}

\Mdef{\eps}{\epsilon}

\renewcommand{\Et}{\cE_t}

\newcommand{\Rdot}{R^{\lrcorner}}
\newcommand{\Mdot}{M^{\lrcorner}}
\newcommand{\M}{\bM}

\newcommand{\tensorR}{\otimes_R}

\newcommand{\modcat}[1]{\mbox{$#1$-mod}}

\newcommand{\cKcellM}{\mbox{$\cK$-cell-$\bM$}}

\newcommand{\Rmod}{\modcat{R}}

\newcommand{\Rdotmod}{\modcat{\Rdot}}

\newcommand{\cEi}{\cE^{-1}}

\newcommand{\Ho}{\mathrm{Ho}}

\Mdef{\SSS}{R}
\Mdef{\RR}{T}

\begin{document}
\title{The Cellularization Principle for Quillen adjunctions} 

\author{J.~P.~C.~Greenlees}
\email{j.greenlees@sheffield.ac.uk}
\address{Department of Pure Mathematics, The Hicks Building, 
Sheffield S3 7RH. UK.}

\author{B.~Shipley}
\email{shipleyb@uic.edu}
\thanks{The first author is grateful for support under EPSRC grant
  number EP/H040692/1.  This material is based upon work by the second author supported by the National Science Foundation under Grant No. DMS-1104396.
}
\address{Department of Mathematics, Statistics and Computer Science, University of Illinois at
Chicago, 508 SEO m/c 249,
851 S. Morgan Street,
Chicago, IL, 60607-7045, USA}



\begin{abstract}
The Cellularization Principle states that under rather weak
conditions, a Quillen adjunction of stable model categories induces a Quillen equivalence on
cellularizations provided there is a derived equivalence on cells. We
give a proof together with a range of examples. 
\end{abstract}


\maketitle

\section{Introduction}
The purpose of this paper is to publicize a useful general principle
when comparing model categories: whenever one has a Quillen
adjunction 
$$\adjunction{F}{\M}{\N}{U}$$
comparing two stable model categories, we obtain another Quillen
adjunction by cellularizing the two model categories with
respect to corresponding objects. Furthermore we obtain a Quillen {\em equivalence}
provided the cells are small and the derived unit or counit is an equivalence on cells.  
In this case, the cellularization of the adjunction induces
a homotopy category level equivalence between the respective
localizing subcategories. The hypotheses are mild, and the 
statement  may appear like a tautology.  The Cellularization Principle can be
directly compared to another extremely
powerful formality, that a natural transformation of cohomology
theories that is an isomorphism on spheres is an equivalence.


This result was first proved in an appendix of the original versions of \cite{tnq3},  but the range of 
cases where the conclusion is useful led us to present the result
separately from that particular application.  

The paper is layed out as follows: in Section \ref{app:cell} we give the statement and proof of the
Cellularization Principle, and  the following sections give  a
selection of examples. 
\section{Cellularization of model categories}\label{app:cell}

Throughout the paper we need to consider models for categories of
cellular objects, thought of as built from a set of basic cells using
coproducts and cofibre sequences. These models are usually obtained
by the process of cellularization (sometimes known as colocalization
or right localization) of model categories, with the
cellular objects appearing as the cofibrant objects. Because it is 
fundamental to our work, we recall some of the basic definitions from~\cite{hh}.

\begin{definition}\cite[3.1.8]{hh}\label{def.k.eq}
Let $\M$ be a model category and $\cK$ be a set of objects in $\M$.
A map $f: X \to Y$ is a {\em $\cK$-cellular equivalence}
if for every element $A$ in $\cK$ the induced map of homotopy function complexes~\cite[17.4.2]{hh} 
$f_*: \map(A, X) \to \map(A, Y)$ is a weak equivalence. 
An object $W$ is {\em $\cK$-cellular} if $W$ is cofibrant in $\M$ and 
$f_*: \map(W, X) \to \map(W, Y)$ is a weak equivalence for
any $\cK$-cellular equivalence $f$.
\end{definition}

One can cellularize a right proper model category under very mild finiteness
hypotheses. To avoid confusion due to the dual use of the word
``cellular''  we recall that a {\em cellular} model
category is a cofibrantly generated model category with smallness
conditions on its generating cofibrations and acyclic cofibrations 
\cite[12.1.1]{hh}.

\begin{proposition}\cite[5.1.1]{hh}\label{prop 5.5}
Let $\M$ be a right proper, 
cellular model category   
and let $\cK$ be a set of objects in $\M$.  The $\cK$-cellularized
model category $\cKcellM$ exists: it has the same underlying category
as $\M$, its weak equivalences are the $\cK$-cellular equivalences,
the fibrations the same as in the original model structure on $\M$, and the cofibrations
are the maps with the left lifting property with respect
to the trivial fibrations. 
The cofibrant objects are the $\cK$-cellular objects.
\end{proposition}

\begin{remark} \label{rem-cell-qe} Since the $\cK$-cellular equivalences are defined using
homotopy function complexes, the $\cK$-cellularized model category
$\cKcellM$ depends only on the homotopy type of the objects in $\cK$.   
\end{remark}

It is useful to have the following further characterization of the
cofibrant objects.   

\begin{proposition}\cite[5.1.5]{hh}\label{prop-cell-obj}
If $\cK$ is a set of cofibrant objects in $\M$, then the class of $\cK$-cellular objects 
agrees with the smallest class of cofibrant ojects in $\M$ that
contains $\cK$ and is closed under homotopy colimits and weak
equivalences.  
\end{proposition}

Throughout this paper we consider stable cellularizations of stable model categories.  Say that a set $\cK$ is {\em stable} if for any $A \in \cK$ all of its suspensions (and desuspensions) are also in $\cK$ up to weak equivalence.   That is, since the cellularization only depends on the homotopy type of elements in $\cK$, if $A \in \cK$, then for all $i \in \Z$ there are objects $B_{i} \in \cK$ with $B_{i} \simeq \Sigma^i A$. In this case, for $\M$ a stable model category and $\cK$ a stable set of  objects, $\cKcellM$ is again a stable model category; see~\cite[4.6]{BR.stable}.  In this case, one can use homotopy classes of maps instead of the homotopy function complexes in Definition~\ref{def.k.eq}.  That is, a map $f: X \to Y$ is a $\cK$-cellular equivalence if and only if for every element $A$ in $\cK$ the induced map $[A, X]_{*} \to [A, Y]_{*}$ is an isomorphism; see~\cite[4.4]{BR.stable}.

\begin{proposition}\label{prop-gen}
If $\M$ is a right proper, stable, cellular model category and $\cK$ is stable, then 
$\cK$ detects trivial objects.  In other words,  an object $X$ is
trivial  in $\Ho(\cKcellM)$, if and only if for each element $A$ in
$\cK$, $[A, X]_{*}=0$, and this group of graded morphisms can equally
be calculated in the homotopy category of $\M$.
\end{proposition}

\begin{proof}
By~\cite[7.3.1]{hovey-model}, the set of cofibres of the generating cofibrations detects trivial objects.  In this stable situation, a set of generating cofibrations is produced as follows in~\cite[4.9]{BR.stable}.  Define a set of horns on $\cK$ by $\Lambda \cK = \{X \otimes \partial\Delta[n]_{+} \to X\otimes \Delta[n]_{+} | n \geq 0, X \in \cK\}$ where here $\otimes$ is defined using framings (see~\cite{BR.stable}, \cite{hovey-model}, or \cite{hh}).  If $J$ is the set of generating acyclic cofibrations in $\M$, then $J \cup \Lambda \cK$ is the set of generating cofibrations for $\cKcellM$. Since $\cK$ is stable, the cofibres of these maps are either contractible or are weakly equivalent to objects in $\cK$ again.  
\end{proof}

Under a finiteness condition $\cK$ is also a set of generators.
An object $K$ is  small in the homotopy category (from now on simply {\em
  small\footnote{some authors use `compact' for this notion}}) if,  for any set of objects
$\{Y_{\alpha}\}$, the natural map $\bigoplus_{\alpha} [K,
Y_{\alpha}]\lra [K,\bigvee_{\alpha}Y_{\alpha}]$ is an isomorphism. 

\begin{corollary}\cite[2.2.1]{ss2} \label{cor-gen}
If $\M$ is a right proper, stable, cellular model category and $\cK$ is a stable
set of small objects, then $\cK$ is a set of generators of $\Ho(\cKcellM)$.  That is,
the only localizing subcategory containing $\cK$ is $\Ho(\cKcellM)$ itself.  
\end{corollary}

Our main theorem states that given a Quillen pair,  appropriate cellularizations of
model categories preserve Quillen adjunctions and induce Quillen
equivalences.   

\begin{theorem}\label{prop-cell-qe} {\bf (The Cellularization Principle.)}
Let $\M$ and $\N$ be right proper, stable, 
cellular model
categories with $F: \M \to \N$ a left Quillen 
functor with right adjoint $U$.  Let $Q$ be a cofibrant replacement functor in $\M$ and
$R$ a fibrant replacement functor in $\N$. 
\begin{enumerate}
\item Let $\cK= \{ A_{\alpha} \}$ be a set of objects in $\M$ with $FQ\cK = \{ FQ A_{\alpha}\}$ the corresponding set in $\N$.  
Then $F$ and $U$
induce a Quillen adjunction 
\[ \adjunction{F}{\cKcellM}{\mbox{$FQ\cK$-cell-$\N$}}{U}\]
between the $\cK$-cellularization
of $\M$ and the $FQ\cK$-cellularization of $\N$.

\item If $\cK= \{ A_{\alpha} \}$ is a stable set of small objects in
  $\M$ such that for each $A$ in $\cK$ the object $FQA$ is small in $\N$ and 
the derived unit $QA \to URFQA$ is a weak equivalence in $\M$, 
then $F$ and $U$ induce a Quillen equivalence between the cellularizations:
\[ \cKcellM \simeq_{Q} \mbox{$FQ\cK$-cell-$\N$}. \]

\item If  $\mcL = \{B_{\beta}\}$ is a stable set of small objects in
  $\N$ such that for each $B$ in $\mcL$ the object $URB$ is small in $\M$ and  
the derived counit  $FQURB \to RB$ is a weak equivalence in $\N$,  then $F$ and $U$
induce a Quillen equivalence between the cellularizations:
\[ \mbox{$UR\mcL$-cell-$\M$} \simeq_{Q} \mbox{$\mcL$-cell-$\N$}.\]
\end{enumerate} 
\end{theorem}

\begin{proof}
Using the equivalences in~\cite[3.1.6]{hh}, 
the criterion in~\cite[3.3.18(2)]{hh} (see also~\cite[2.2]{hovey-stable}) 
for showing that $F$ and $U$ 
induce a Quillen adjoint pair on the cellularized model categories in $(1)$
is equivalent to requiring that $U$ takes  $FQ\cK$-cellular equivalences between fibrant objects to $\cK$-cellular equivalences.
Any Quillen adjunction induces
a weak equivalence $\map(A, URX) \simeq \map(FQA, X)$
of the homotopy function complexes, see for example~\cite[17.4.15]{hh}.
So a map $f: X \to Y$ induces a weak equivalence
$f_*: \map(FQA, X) \to \map(FQA, Y)$ if and only if
 $Uf_*: \map(A, URX) \to \map(A, URY)$       
is a weak equivalence.  Thus in (1), $U$ preserves (and reflects) the cellular equivalences between fibrant objects. Hence, $U$ induces a Quillen adjunction on the cellularized model categories.

Similarly, $Uf_*: \map(URB, URX) \to \map(URB, URY)$       
is a weak equivalence if and only if $f_*: \map(FQURB, X) \to \map(FQURB, Y)$ is.
Given the hypothesis in $(3)$ that $FQURB \to RB$ is a weak equivalence, it follows that $Uf_{*}$ is a weak equivalence if and only if $f_{*}: \map(B, X) \to \map(B, Y)$ is.  Thus,
 it follows in $(3)$ that $U$ preserves (and reflects) the cellular equivalences between fibrant objects. Hence, $U$ induces a Quillen adjunction on the cellularized model categories.
 Note that the stability of $\M$, $\N$, $\cK$ and $\mcL$ was not necessary for establishing the Quillen adjunction in (1) or (3).

We establish (2) in the next paragraph. One can make very similar arguments for (3) or one can deduce (3) from (2).  To deduce (3) from (2), consider (2) applied to $\cK = UR\mcL$. The hypothesis in (3) implies the hypothesis in (2) and thus produces a Quillen equivalence between $UR\mcL$-cell-$\M$ and $FQUR\mcL$-cell-$\M$.  The hypothesis in (3) also implies that $FQUR\mcL$-cell-$\M$ and $\mcL$-cell-$\M$ are the same cellularization of $\M$.  Thus, (3) follows.  

We now return to the Quillen equivalence in (2). Since $\M$ and $\cK$ are stable, $\cKcellM$ is a stable model category by~\cite[4.6]{BR.stable}.  Since left Quillen functors preserve homotopy cofibre sequences, $FQ\cK$, and hence also $FQ\cK$-cell-$\N$, are stable.  
The Quillen adjunction in (1) induces a derived adjunction on the triangulated homotopy categories; we show that this is actually a derived equivalence.    
Both derived functors are exact (since the left adjoint commutes with suspension and cofibre sequences and the right adjoint commutes with loops and fibre sequences).  As a left adjoint, $F$ also preserves coproducts.  We next show that the right adjoint preserves coproducts as well.

Since $\cK = \{A_{\alpha}\}$ detects $\cK$-cellular equivalences, to show that $U$ preserves coproducts it suffices to show that for each $A_{\alpha} \in \cK$ and any family $\{X_{i}\}$ of objects in $\N$ the natural map 
$$[A_{\alpha}, \bigvee_i UX_{i}] \to [A_{\alpha}, U(\bigvee_i X_{i})]$$ is an isomorphism.  Using the adjunction and the fact that each $A_{\alpha}$ is small, the source can be rewritten as  
$$[A_{\alpha}, \bigvee_i UX_{i}] \iso \bigoplus_{i} [A_{\alpha}, UX_{i}] \iso \bigoplus_{i} [FQA_{\alpha}, X_{i}] .$$
Similarly, using the adjunction, the target is isomorphic to $[FQA_{\alpha}, \bigvee_i X_{i}]$.  Since $FQA_{\alpha}$ is assumed to be small, the source and target are isomorphic and this shows that $U$ commutes with coproducts.

Consider the full subcategories of objects $M$ in $\Ho(\cKcellM)$ and $N$ in $\Ho(\mbox{$FQ\cK$-cell-$\N$})$
such that the unit $QM \to URFQM$ or counit $FQURN \to RN$ of the adjunctions are equivalences.   Since both derived functors are exact and preserve coproducts, these are localizing subcategories.
Since for each $A$ in $\cK$ the unit is an equivalence and $\cK$ is
a set of generators by Corollary~\ref{cor-gen}, the unit is an equivalence on all of $\Ho(\cKcellM)$.   
It follows that the counit is also an equivalence for each object $N = FQA$ in $FQ\cK$.
Since $FQ\cK$ is a set of generators for $\Ho(\mbox{$FQ\cK$-cell-$\N$})$, the counit is 
also always an equivalence.   Statement (2) follows.
\end{proof}

Note that if $F$ and $U$ form a Quillen equivalence on the original categories,
then the conditions in Theorem \ref{prop-cell-qe} parts (2) and (3) are automatically
satisfied.   Thus, they also induce Quillen equivalences on the
cellularizations.   

\begin{corollary}\label{cor-cell-qe}
Let $\M$ and $\N$ be right proper, stable 
cellular model
categories with $F: \M \to \N$ a Quillen 
equivalence with right adjoint $U$.   Let $Q$ be a cofibrant replacement functor in $\M$ and
$R$ a fibrant replacement functor in $\N$. 
\begin{enumerate}
\item Let $\cK= \{ A_{\alpha} \}$ be a stable set of small objects in $\M$, with $FQ\cK = \{ FQ A_{\alpha}\}$ the corresponding
set of objects in $\N$. 
Then $F$ and $U$
induce a Quillen equivalence between the $\cK$-cellularization
of $\M$ and the $FQ\cK$-cellularization of $\N$:
\[ \mbox{\cKcellM}\simeq_Q \mbox{$FQ\cK$-cell-$\N$} \]
\item Let $\mcL = \{B_{\beta}\}$ be a set of small objects in $\N$, with $UR\mcL = \{ URB_{\beta}\}$ the corresponding
set of objects in $\N$.   Then $F$ and $U$
induce a Quillen equivalence between the $\mcL$-cellularization
of $\N$ and the $UR\mcL$-cellularization of $\M$:
\[ \mbox{$UR\mcL$-cell-$\M$} \simeq_{Q} \mbox{$\mcL$-cell-$\N$}\]
\end{enumerate} 
\end{corollary}

In~\cite[2.3]{hovey-stable} Hovey gives criteria for when
localizations preserve Quillen equivalences.  Since
cellularization is dual to localization,  a generalization of this corollary without stability or smallness hypotheses follows from the 
dual of Hovey's statement.   

\section{Smashing localizations}
We suppose given a map $\theta : R \lra T$ of ring spectra (or DGAs). This gives
the extension and restriction of scalars Quillen  adjunction 
$$\adjunction{\theta_*}{\mbox{$R$-mod}}{\mbox{$T$-mod}}{\theta^*},   $$
where $\theta_*N=T\sm_{R} N$.
We apply Theorem~\ref{prop-cell-qe} with $\M=\mbox{$R$-mod}$ and $\N=\mbox{$T$-mod}$. 
The category of $T$-modules is generated by the $T$-module $T$, and we
use that as the generating cell. The following uses the ideas of
the Cellularization Principle. 

\begin{corollary}\label{cor:smashing}
If $T\sm_{R}T \stackrel{\simeq }\lra T$ 
is an equivalence of $R$-modules, then the Quillen pair 
$$\adjunction{\theta_*}{\mbox{$T$-cell-$R$-modules}}
{\mbox{$T$-modules}}{\theta^*}$$
induces 

\begin{enumerate}
\item a Quillen equivalence if $\theta^* T$ is small as an $R$-module, or
\item in general, an equivalence of triangulated categories
$$\mbox{$T$-loc-Ho($R$-modules)}\simeq
\mbox{Ho($T$-modules)}$$
where $loc$ denotes the localizing subcategory. 
\end{enumerate}
\end{corollary}

\begin{proof}
In the first case with $\theta^* T$ small, this follows directly from
 Part 3,Theorem~\ref{prop-cell-qe} with $\M=\mbox{$R$-modules}$ and
$\N=\mbox{$T$-modules}$, taking $T$ to be the generator of $\N$. 
In the second case, we again apply Part 3.
Here the hypothesis shows that the counit is a derived equivalence on
cells.  However the  complication is that
$\theta^*T$ will  not usually be small as an $R$-module. 

Nonetheless, the counit is still
a derived equivalence for the $R$-module $ T$. It remains to argue
that the derived counit gives the stated equivalence. 
For this we note that the right adjoint preserves arbitrary sums: this is obvious if we
are working with actual modules, but in general we may use the fact
that $\theta^*$ is also a left adjoint (with right adjoint the
coextension  of scalars). It follows that we have an equivalence of
the localizing subcategories generated by $T$ on the two sides. 
\end{proof}

For the first example, we take $R$ and $T$ to be conventional commutative
rings or DGAs and
$T=\cEi R$ for some multiplicatively closed set $\cE$. The condition
is  satisfied since $\cEi R\otimes_{R} \cEi R\cong \cEi R$ and we find
$$\mbox{$\cEi R$-loc-Ho($R$-modules)}\simeq
\mbox{Ho($\cEi R$-modules)}. $$

\begin{example}
It is worth giving an example to show that we do not obtain a Quillen
equivalence between $T$-cell-$R$-modules and $T$-modules in general. 
This shows  that the smallness hypothesis in the Celllularization Principle is
necessary. 

For this we take $R=\Z$ and $T=\Z [1/p]$. We note that any object $M$  in
the localizing subcategory of $T$-cell-$R$-modules generated by $T$ has the
property  that $M\simeq M[1/p]$. Accordingly $M=\Z /p^{\infty}$ is not
in this localizing subcategory. On the other hand $M$ is not
$T$-cellularly equivalent to 0 since $\Hom (\Z [1/p],
\Z/p^{\infty})\neq 0$.
\end{example}

More generally any smashing localization of module spectra behaves
in a similar way (see~\cite{jjg} for related discussion).
We first suppose that $R$ is a cofibrant $\bS$-algebra spectrum in
the sense of~\cite{ekmm}, where $\bS$ is the sphere spectrum; see also Remark~\ref{rem.cof} below.
We next suppose given a smashing localization $L$ of  the category of 
$R$-module spectra.  By ~\cite[VIII.2.2]{ekmm}, there is a ring map
$\theta: R\lra LR$ so that we
may apply the above discussion with $T=LR$.  The smashing condition 
states that we have an equivalence $LN \simeq LR \sm_{R} N$ for any
$R$-module $N$. 

\begin{proposition}\label{prop.3.2}
Suppose given a cofibrant $\bS$-algebra spectrum $R$ and a smashing localization $L$ on the
category of $R$-modules.  We have a diagram 
$$\diagram
\mbox{$R$-modules} \rto \dto &\mbox{$LR$-modules}\\
L(\mbox{$R$-modules})\urto_{\simeq}&
\enddiagram$$
of left Quillen functors where $L(\mbox{$R$-modules})$ is the localization of the model category of
$R$-modules. These induce a  Quillen equivalence
$$ L(\mbox{$R$-modules})\simeq_Q \mbox{$LR$-modules},  $$
and triangulated equivalences 
$$ \mbox{$LR$-loc-Ho (\mbox{$R$-modules})}\simeq
Ho( L(\mbox{$R$-modules}))\simeq \mbox{Ho($LR$-modules)}.  $$
\end{proposition} 

\begin{proof}
The ring map $\theta: R\lra LR=T$ gives the top horizontal left
Quillen functor $\theta_*$. The vertical map is the identity on
underlying categories, which is a left Quillen functor by definition
of the local model structure. The diagonal map is again $\theta_*$, and it
exists because $L$ is smashing so that $L$-equivalences are taken to
equivalences of $LR$-modules. This gives the diagram of left Quillen
functors. 

To see the diagonal is a Quillen equivalence we apply Part 3 of the 
Cellularization Principle. The $LR$-module $LR $ is a small generator 
of $LR$-modules, and the smashing condition means
that  $L(\mbox{$R$-modules})$ is generated by the single object $LR$,
and  the universal property together with the smallness of $R$ shows
$LR$ is small in $L(R-modules)$. 
 
We apply Corollary \ref{cor:smashing} to obtain the statement about 
localizing categories, since the  smashing condition applies to the generator 
$N=LR$ to show the hypotheses hold. 

See also~\cite[VIII.3.2]{ekmm} and~\cite[3.2(iii)]{ss2}. 
\end{proof}

\begin{remark}\label{rem.cof}
For an arbitrary $\bS$-algebra spectrum $T$, consider a cofibrant
replacement $\theta : R \to T$ in $\bS$-algebras and consider the
relationship between smashing localizations of $R$-modules and $T$-modules.
More precisely, suppose given a bifibrant $T$-module $E$ so that $L_E$
(localization with respect to the
$E$-equivalences as in~\cite[VIII.1.1]{ekmm}) is a smashing localization
on $T$-modules.  In this case there is a
corresponding $R$-module $F$, giving a smashing localization $L_F$ on 
$R$-modules such that the localized model categories are Quillen equivalent,
$$ L_F(R\mbox{-modules}) \simeq_Q L_E(T\mbox{-modules}).$$

In fact one can take $F$ to be the cofibrant replacement in
$R$-modules of $\theta^*E$.  Since $\theta :R \to T$ is a weak equivalence, the functors $\theta_*$ and $\theta^*$ induce a Quillen equivalence between the categories of $R$-modules and $T$-modules.  Using the criteria in~\cite[2.3]{hovey-stable} for when localizations preserve Quillen equivalences (dual to Corollary~\ref{cor-cell-qe} above), one can show that the localization model category of $R$-modules with respect to the $F$-equivalences is Quillen equivalent to the localization model category of $T$-modules with respect to the $\theta_* F \iso T \sm_R F$-equivalences.  Since $\theta_*F$ and $E$ are weakly equivalent cofibrant $T$-modules, the $\theta_*F$-equivalences agree with the $E$-equivalences, so we have
$$ L_F(R\mbox{-modules}) \simeq_Q L_{\theta_*F}(T\mbox{-modules})\iso L_E(T\mbox{-modules}).$$
In this situation one can show that there is a weak equivalence $T \sm_R L_F R \simeq L_E T$.  Using this, one can then show that if $L_E$ is smashing, then $L_F$ is also a smashing localization.  Note here though that~\cite[VIII.2.2]{ekmm} only applies to $R$.  So although $L_F R$ is constructed as an $S$-algebra, $L_E T$ is only constructed as a $T$-module.  
\end{remark}

Perhaps the best known example in the category of spectra is when we consider the localization  of $\bS$-modules with respect to $E_n,$ the $n$th $p$-local Morava $E$-theory. We denote localization with respect to any spectrum weakly equivalent to $E_{n}$ by $L_{n}$. Following the remark above, we consider $R = c\bS$, the cofibrant replacement of the sphere spectrum $\bS$ and see that $L_{n}(\bS$-modules$)$ and $L_{n}(c\bS$-modules$)$ are Quillen equivalent.  Proposition~\ref{prop.3.2} then shows that
$$
L_n(\mbox{$\bS$-modules}) 
\simeq_{Q}
L_n(\mbox{$c\bS$-modules})
\simeq_{Q} 
\mbox{$(L_n c\bS)$-modules} 
 $$
and the homotopy categories of these are equivalent to 
$\mbox{$(L_n \bS)$-loc-Ho($\bS$-modules)}$.

\section{Isotropic equivalences of ring $G$-spectra}
We suppose given a map $\theta : R \lra T$ of ring $G$-spectra. This gives
the extension and restriction of scalars Quillen  adjunction 
$$\adjunction{\theta_*}{\mbox{$R$-mod}}{\mbox{$T$-mod}}{\theta^*}.   $$
We apply Theorem~\ref{prop-cell-qe} Part 2 with $\M=\mbox{$R$-mod}$ and $\N=\mbox{$T$-mod}$. 
If $G$ is the trivial group, then $R$ generates $R$-modules and
the Cellularization Principle shows we have an equivalence if $\theta$
is a weak equivalence of $R$-modules. 

If $G$ is non-trivial,  we get a somewhat more
interesting example. The category of $R$-modules is generated by the
extended objects $G/H_+ \sm R$ as $H$ runs through closed subgroups of
$G$ and the unit is the comparison $G/H_+ \sm \theta$.  
If $\cF$ is a family of subgroups, we say $\theta $ is an $\cF$-equivalence if $G/H_+ \sm \theta$ is an equivalence for all $H$ in  $\cF$.  Define $\cF$-cellularization of $R$-mod to be cellularization
with respect to the set of all suspensions and desuspensions of
objects $G/H_+ \sm R$ for $H$ in  $\cF$. Then the
Cellularization Principle shows that if $\theta$ is an
$\cF$-equivalence we have a Quillen equivalence
$$\mbox{$\cF$-cell-$R$-module-$G$-spectra}\simeq
\mbox{$\cF$-cell-$T$-module-$G$-spectra}. $$

\section{Torsion modules}
Let $R$ be a conventional commutative Noetherian ring and $I$ an
ideal.  We apply Theorem~\ref{prop-cell-qe}, with $\N$ the
category of differential graded $R$-modules and $\M$ 
 the category of differential graded $I$-power torsion modules.  
There is an adjunction 
$$\adjunction{i}{ \mbox{$I$-power-torsion-$R$-modules}}{\mbox{$R$-modules}}{\Gamma_I}$$
with left adjoint $i$ the inclusion  and the right adjoint $\Gamma_I$ defined by 
$$\Gamma_I(M)=\{ m \in M\st I^Nm=0 \mbox{ for } N>>0\}. $$
Both of these categories support injective model structures
by~\cite[2.3.13]{hovey-model}, with cofibrations the monomorphisms and
weak equivalences the quasi-isomorphisms.  For torsion-modules, one
needs to bear in mind that  to construct products and inverse limits one forms them in the category of all $R$-modules and then applies the right adjoint $\Gamma_I$; see also~\cite[8.6]{gfreeq}.  With these structures the above adjunction is a Quillen adjunction.

We now consider the Cellularization Principle with $\M=\mbox{$I$-power-torsion-$R$-modules}$ and 
$\N=\mbox{$R$-modules}$. If $I=(x_1, x_2, \ldots, x_n)$, we may form
the  Koszul complex $K:=K(x_1, x_2, \ldots , x_n)$ as the tensor product
of the complexes $R\stackrel{x_i}\lra R$, noting that it is small by
construction and therefore suitable for use as a cell.  Since the
homology of $K$ is $I$-power torsion, $K$ is equivalent to
an object $K'$ in the category of $I$-power torsion modules, which we
may take to be fibrant. We now apply the Cellularization
Principle to give an equivalence 
$$ \mbox{$\Gamma_I K'$-cell-$I$-power-torsion-$R$-modules}\simeq
\mbox{$K'$-cell-$R$-modules}. $$

It is  proved in \cite[6.1]{tec}  that the localizing subcategory 
generated by $R/I$  is also generated by $K\simeq K'$.
By the same proof,  we see that
$K'=\Gamma_IK'$ generates the category of $I$-power torsion modules
and  we conclude
$$ \mbox{$I$-power-torsion-$R$-modules}\simeq
\mbox{$R/I$-cell-$R$-modules}. $$

\section{Hasse equivalences}
The idea here is that if a ring (spectrum or differential graded algebra) $R$ is expressed as the pullback of a
diagram of rings,  the Cellularization Principle lets us build up the
model category of  differential graded $R$-modules from categories of
modules over the terms.  See also~\cite{tnq.diagrams} for a more
general treatment. We apply the standard context of
Theorem~\ref{prop-cell-qe} with $\M$ 
 the category of $R$-modules. 

\subsection{Diagrams of modules}
To describe $\N$ we start with a commutative diagram 
$$\diagram
R\rto^{\alpha} \dto_{\beta} & R^l \dto^{\gamma}\\
R^c \rto_{\delta}&R^t.
\enddiagram$$
of rings.  
\begin{example} The classical Hasse principle is built on the pullback square
$$\diagram 
\Z \rto \dto &\Q \dto\\
\prod_p\Z_p^{\wedge}\rto &(\prod_p\Z_p^{\wedge})\otimes \Q
\enddiagram$$
\end{example}

Returning to the general case, we delete $R$ and consider the diagram 
$$
\Rdot =\left(
\diagram
& R^l \dto^{g}\\
R^c \rto_{d}&R^t
\enddiagram 
\right) $$
with three objects.
We may form the category $\N=\Rdotmod$ of diagrams
$$\diagram
& M^l \dto^{h}\\ 
M^c \rto_{e}&M^t
\enddiagram$$ 
where $M^l$ is an $R^l$-module, $M^c$ is an $R^c$-module, 
$M^t$ is an $R^t$-module and the maps $h$ and $e$ 
are module maps over the corresponding maps of rings.  That is, $h: M^{l} \to g^{*}M^{t}$ is a map of $R^{l}$-modules and $e:M^{c} \to d^{*}M^{t}$ is a map of $R^{c}$-modules. We will return
to model structures below.

\subsection{An adjoint pair}
 Since $\Rdot$ is a diagram of $R$-algebras,  termwise tensor product
gives a  functor
$$\Rdot \tensorR (\cdot ):\Rmod \lra \Rdotmod . $$
Similarly, since $R$ maps to the pullback $P\Rdot$, pullback 
gives a functor
$$P: \Rdotmod\lra \Rmod.$$
It is easily verified that these give an adjoint pair
$$\adjunction{\Rdot\tensorR (\cdot )}{\Rmod}{\Rdotmod}{P}. $$

We may then consider the derived unit
$$\eta: M \lra \underline{P}(\Rdot \tensorR^{L} M).  $$
Since $R$ is the generator of the category of $R$-modules, we want to 
require that $\eta$ is an equivalence
when $M=R$, which is to say the original diagram of rings is 
a homotopy pullback.  In fact, we first fibrantly replace the original diagram of rings in the diagram-injective model category of pull-back diagrams of rings.  See~\cite[5.1.3]{hovey-model} or ~\cite[15.3.4]{hh}.   This is discussed in more detail in~\cite{tnq.diagrams}.

On the other hand, we cannot expect the counit of the adjunction
to be an equivalence since we can add any module to  $M^t$ without changing $P\Mdot$.
This is where the Cellularization Principle comes in. We should use
the image of $R$  to cellularize the category of diagrams of
modules. In preparation for this, we describe the model structure.

\subsection{Model structures}
We give categories of (differential graded) modules over a ring the (algebraically) projective model 
structure, with homology isomorphisms as weak equivalences and
fibrations the surjections. The cofibrations are retracts of relative
cell complexes, where the spheres are shifted copies of $R$.
The category  $\Rdotmod$ gets the diagram-injective
model structure in which cofibrations
and weak equivalences are maps which have this property objectwise; 
the fibrant objects have $\gamma$ and $\delta$ surjective.  
This diagram-injective model structure is shown to exist for more general diagrams of ring
spectra in an appendix of the original versions of \cite{tnq3}, see also \cite{tnq.diagrams}, and the same proof works for DGAs.  

\subsection{The Quillen equivalence}
Since extension of scalars is a left Quillen functor for the (algebraically)
projective model structure for any map of DGAs,  $\Rdot \tensorR -$ 
preserves cofibrations and weak equivalences and is therefore also a 
left Quillen functor.   We then apply the Cellularization Principle
 to obtain the following result. 

\begin{proposition} 
Assume given a commutative square of DGAs which is a 
homotopy pullback. The adjunction induces a Quillen equivalence
$$\Rmod \stackrel{\simeq}\lra \mbox{$\Rdot$-cell-$\Rdot$-mod}, $$
where cellularization is with respect to the image, $\Rdot$, of the
generating $R$-module $R$.
\end{proposition}

\begin{proof} 
We apply Theorem \ref{prop-cell-qe}, which states that if we
cellularize the model categories with respect to corresponding sets  
of small objects, we obtain a Quillen adjunction. 

In the present case, we cellularize with respect to the single small $R$-module $R$ on the left, and the corresponding diagram $\Rdot$ on the right.  First we verify that $\Rdot$ is small.  Consider the three evaluation functors from $\Rdot$-modules down to modules over the rings $R^{l}$, $R^{c}$, or $R^{t}$ and the associated left adjoints of these evaluation functors $L^{l}$, $L^{c}$, and $L^{t}$.  The $\Rdot$-module $\Rdot$ is the pushout of the following diagram.
$$\diagram
& L^{l}R^l \\ 
L^{c}R^c &L^{t}R^{t} \lto \uto
\enddiagram$$ 
Indeed, one may explicit: 
 $L^lR^l$ is $\Rdot$ with $R^c $ replaced by 0,  
 $L^cR^c$ is $\Rdot$ with $R^l $ replaced by 0, and 
 $L^tR^t$ is $\Rdot$ with $R^l$ and $R^c $ replaced by 0. 
This diagram is also a homotopy pushout diagram between objects which are cofibrant in the diagram-injective model structure on $\Rdot$-modules.   It follows that $\Rdot$ is small in the homotopy category of $\Rdot$-modules.

Since the original  diagram of rings is a homotopy
pullback,  the unit of the adjunction is an equivalence for $R$,
and we see that the generator $R$ and the generator $\Rdot$ correspond
under the equivalence, as required in the hypothesis in Part (2) of
Theorem \ref{prop-cell-qe}.   

Since $R$ is cofibrant and generates $\Rmod$, cellularization with
respect to $R$ has no effect on  $\Rmod$ and we obtain the stated  equivalence
with the cellularization of $\Rdotmod$ with respect to the diagram
$\Rdot$. 
\end{proof}

\end{document}